\documentclass[a4paper,reqno,10pt]{article}

\usepackage{epsf,graphicx,epsfig}
\usepackage{amscd}
\usepackage{amsmath,latexsym,amssymb,amsthm}
\usepackage[nospace,noadjust]{cite}
\usepackage{textcomp}
\usepackage{setspace,cite}
\usepackage{lscape,fancyhdr,fancybox}
\usepackage{stmaryrd}
\usepackage[all,cmtip]{xy}
\usepackage{tikz}
\usepackage{cancel}
\usetikzlibrary{shapes,arrows,decorations.markings}
\usepackage{graphicx}

\usepackage{color}
\usepackage{url}
\usepackage{enumerate}
\usepackage[mathscr]{euscript}

  \theoremstyle{plain}

\swapnumbers
    \newtheorem{thm}{Theorem}[section]
    \newtheorem{prop}[thm]{Proposition}

    \newtheorem{corollary}[thm]{Corollary}
    
    \newtheorem{subsec}[thm]{}
\theoremstyle{definition}
    \newtheorem{defn}[thm]{Definition}
        \newtheorem{remark}[thm]{Remark}
\theoremstyle{remark}

\title{}
\author{}
\date{}
\usepackage{amssymb}

\usepackage{hyperref}
\hypersetup{
	colorlinks,
	citecolor=blue,
	filecolor=black,
	linkcolor=blue,
	urlcolor=black
}

\begin{document}

\title{Cohomology and deformations of weighted Rota-Baxter operators}

\author{Apurba Das}

\maketitle

\begin{center}
Department of Mathematics and Statistics,\\
Indian Institute of Technology, Kanpur 208016, Uttar Pradesh, India.\\
Email: apurbadas348@gmail.com
\end{center}



\begin{abstract}
Weighted Rota-Baxter operators on associative algebras are closely related to modified Yang-Baxter equations, splitting of algebras, weighted infinitesimal bialgebras, and play an important role in mathematical physics. For any $\lambda \in {\bf k}$, we construct a differential graded Lie algebra whose Maurer-Cartan elements are given by $\lambda$-weighted relative Rota-Baxter operators. Using such characterization, we define the cohomology of a $\lambda$-weighted relative Rota-Baxter operator $T$, and interpret this as the Hochschild cohomology of a suitable algebra with coefficients in an appropriate bimodule. We study linear, formal and finite order deformations of $T$ from cohomological points of view. Among others, we introduce Nijenhuis elements that generate trivial linear deformations and define a second cohomology class to any finite order deformation which is the obstruction to extend the deformation. In the end, we also consider the cohomology of $\lambda$-weighted relative Rota-Baxter operators in the Lie case and find a connection with the case of associative algebras. 
\end{abstract}

\medskip

\noindent {\sf 2020 MSC classification:} 16E40, 16S80, 17B38.

\noindent {\sf Keywords:} Weighted Rota-Baxter operators, Maurer-Cartan elements, Cohomology, Deformations.

\medskip


\thispagestyle{empty}

\tableofcontents

\vspace{0.2cm}

\section{Introduction}

Rota-Baxter operators (of weight $0$) was first appeared in the work of G. Baxter \cite{baxter} in his study of the fluctuation theory in probability. Such operators can be seen as an algebraic abstraction of the integral operator. Rota-Baxter operators on associative algebras were further studied by G.-C. Rota \cite{rota} and others \cite{atkinson,cartier,miller}. There is a close connection between Rota-Baxter operators and Yang-Baxter equation \cite{aguiar}. In last twenty years, Rota-Baxter operators have found important applications in renormalizations of quantum field theory \cite{connes}, pre-algebras \cite{bai-spl}, infinitesimal bialgebras \cite{aguiar} and double algebras \cite{gon}. See \cite{guo-book} for more on Rota-Baxter operators. The notion of relative Rota-Baxter operators (also called generalized Rota-Baxter operators or $\mathcal{O}$-operators) are a generalization of Rota-Baxter operators in the presence of a bimodule \cite{uchino}. They can be seen as a noncommutative analogue of Poisson structures. Rota-Baxter operators and relative Rota-Baxter operators also appeared in the context of Lie algebras \cite{kuper,guo-lax}. They are related to Rota-Baxter operators on Lie groups via global-infinitesimal correspondence \cite{sheng-adv,sheng-w1}. 

\medskip

Weighted (relative) Rota-Baxter operators are generalization of (relative) Rota-Baxter operators. They are related to post-algebras \cite{guo-lax}, weighted infinitesimal bialgebras \cite{aybe}, weighted associative Yang-Baxter equations \cite{aybe}, combinatorics of planar rooted forests \cite{forest}, and play an important role in mathematical physics \cite{guo-lax}. Some classification result of weighted Rota-Baxter operators on matrix algebras are given in \cite{guba}. See \cite{brz,das-jmp} for some other generalizations of Rota-Baxter operators.

\medskip

On the other hand, the algebraic deformation theory began with the seminal work of M. Gerstenhaber \cite{gers} for associative algebras, followed by its extension to Lie algebras by A. Nijenhuis and R. Richardson \cite{nij-ric}. These deformations are governed by suitable cohomologies (e.g. Hochschild cohomology for associative algebras and Chevalley-Eilenberg cohomology for Lie algebras). See \cite{bala} for deformations of algebras over any binary quadratic operad. Deformations of (relative) Rota-Baxter operators on Lie algebras have been developed in \cite{tang} by introducing a new cohomology theory. It has been extended to associative algebras in \cite{das-rota}.

\medskip

Recently, the authors in \cite{sheng-w1} introduced a cohomology theory for relative Rota-Baxter operators of weight $1$ on Lie algebras and studied their linear deformations. Their cohomology is given by the Chevalley-Eilenberg cohomology of a suitable Lie algebra with coefficients in an appropriate representation. This description of cohomology has lacks of information to know the structure of the cohomology ring and to study formal and finite order deformations. Our aim in this paper is to define the cohomology of a relative Rota-Baxter operator of arbitrary weight (not necessary of weight $1$) using Maurer-Cartan element in a suitable differential graded Lie algebra. This will fulfil the gaps of \cite{sheng-w1}. However, we will mainly focus on associative algebras (the Lie case is also described at the end).

\medskip

Given a fixed scalar $\lambda \in {\bf k}$, we first construct a differential graded Lie algebra whose Maurer-Cartan elements are precisely $\lambda$-weighted relative Rota-Baxter operators on associative algebras. This characterization of $\lambda$-weighted relative Rota-Baxter operators allows us to define cohomology theory for such operators. We also interpret this cohomology as the Hochschild cohomology of a suitable associative algebra with coefficients in an appropriate bimodule. Next, we study various aspects of deformations (linear, formal and finite order deformations) of $\lambda$-weighted relative Rota-Baxter operators. We introduce Nijenhuis elements associated with a $\lambda$-weighted relative Rota-Baxter operator $T$ that generate trivial linear deformations of $T$. We find a sufficient condition for the rigidity of the operator $T$ in terms of Nijenhuis elements. For a finite order deformation of $T$, we also associate a second cohomology class, which is the obstruction to extend the given deformation to a next order deformation.

\medskip

We end this paper by considering $\lambda$-weighted relative Rota-Baxter operators on Lie algebras. We define the cohomology of such operators using Maurer-Cartan characterizations in a suitable differential graded Lie algebra. When $\lambda = 1$, our cohomology coincides with the one given in \cite{sheng-w1}. Finally, we relate it with the cohomology of $\lambda$-weighted relative Rota-Baxter operators on associative algebras by suitable skew-symmetrizations.

\medskip

\noindent {\bf Organization of the paper.}  In Section \ref{sec-2}, we first recall weighted relative Rota-Baxter operators on associative algebras and prove some basic results about such operators. In Section \ref{sec-3}, we define cohomology of weighted relative Rota-Baxter operators using Maurer-Cartan characterizations. We also show that such cohomology can be interpreted as the Hochschild cohomology. Deformations of weighted relative Rota-Baxter operators are considered in Section \ref{sec-4}. Finally, in Section \ref{sec-5}, we focus on weighted relative Rota-Baxter operators on Lie algebras and define their cohomology. Relations with the case of associative algebras are also described.

\medskip

\noindent {\bf Notations.} Let $\big( \mathfrak{g} = \oplus \mathfrak{g}^n, [~,~]_\mathfrak{g}, d  \big)$ be a differential graded Lie algebra. An element $\theta \in \mathfrak{g}^1$ is said to be a Maurer-Cartan element if $\theta$ satisfies
\begin{align*}
d \theta + \frac{1}{2} [\theta, \theta]_\mathfrak{g} = 0.
\end{align*}
We denote by $\mathbb{S}_n$ the set of all permutations on the set $\{1, \ldots, n \}$. A permutation $\sigma \in \mathbb{S}_n$ is called a $(p,q)$-shuffle (with $p+q= n$) if $\sigma (1) < \cdots < \sigma (p)$ and $\sigma (p+1) < \cdots < \sigma (p+q)$. The set of all $(p,q)$-shuffles are denoted by $\mathbb{S}_{(p,q)}$. All vector spaces, (multi)linear maps, tensor products, wedge products are over a field ${\bf k}$ of characteristic zero.

\section{Weighted relative Rota-Baxter operators}\label{sec-2}

In this section, we recall weighted relative Rota-Baxter operators \cite{aybe,forest,uchino} and study some basic properties.  

Let $A$ and $B$ be two associative algebras. We denote the elements of $A$ by $a, b, a_1, a_2, \ldots $ and the elements of $B$ by $u, v, u_1, u_2, \ldots$. We also denote the multiplication in $A$ by $\cdot_A$ and the multiplication in $B$ by $\cdot_B$. Let the associative algebra $A$ acts on $B$. That is, $B$ is an $A$-bimodule \big(that consist of left and right $A$-actions $l: A \otimes B \rightarrow B,~(a, u) \mapsto l_a (u) = a \cdot u$ and $r: B \otimes A \rightarrow B,~ (u,a) \mapsto r_a (u) = u \cdot a$ with the followings $(a \cdot_A b) \cdot u = a \cdot ( b \cdot u)$, $(a \cdot u) \cdot b = a \cdot (u \cdot b)$ and $(u \cdot a) \cdot b = u \cdot ( a\cdot_A b)$ \big) satisfying additionally 
\begin{align*}
(a \cdot u) \cdot_B v = a\cdot (u \cdot_B v), \qquad (u \cdot a) \cdot_B v = u \cdot_B (a \cdot v) ~~~ \text{ and } ~~~ (u \cdot_B v) \cdot a = u  \cdot_B (v\cdot a),
\end{align*} 
for $u, v \in B$  and $a, b \in A.$ In this case, we often say that $B$ is an associative $A$-bimodule. Note that any associative algebra $A$ is an associative $A$-bimodule with left and right $A$-actions are given by the algebra multiplication.

The following result is standard and we omit the details.

\begin{prop}
Let $A$ and $B$ be two associative algebras and $B$ be an associative $A$-bimodule. Then for any $\lambda \in {\bf k}$, the direct sum $A \oplus B$ carries an associative structure given by
\begin{align*}
(a,u) \bullet_\lambda (b,v) = (a \cdot_A b ,~ a \cdot v + u \cdot b + \lambda~ u \cdot_B v), ~ \text{ for } a, b \in A \text{ and } u, v \in B.
\end{align*}
This is called the $\lambda$-weighted semidirect product and denoted by $A \ltimes_\lambda B$.
\end{prop}

\begin{defn}
\begin{itemize}
\item[(i)] Let $A$ be an associative algebra. A linear map $T: A \rightarrow A$ is said to be $\lambda$-weighted Rota-Baxter operator on $A$ if $T$ satisfies
\begin{align}
T(a) \cdot_A T (b) = T \big( T(a) \cdot_A b + a \cdot_A T(b) + \lambda~ a \cdot_A b \big), ~ \text{ for } a, b \in A.
\end{align}
\item[(ii)] Let $A$ and $B$ be two associative algebras and $B$ be an associative $A$-bimodule. A linear map $T : B \rightarrow A $ is said to be a $\lambda$-weighted relative Rota-Baxter operator on $B$ over the algebra $A$ if
\begin{align}\label{rel-rb-id}
T(u) \cdot_A T (v) = T \big( T(u) \cdot v + u \cdot T(v) + \lambda~ u \cdot_A v \big), ~ \text{ for } u, v \in B.
\end{align}
We simply call $T$ as a $\lambda$-weighted relative Rota-Baxter operator when the domain and codomain of $T$ are clear from the context.
\end{itemize}
\end{defn}

\begin{remark}\label{rel-not}
It follows from the above definitions that a $\lambda$-weighted Rota-Baxter operator on $A$ is a particular case of $\lambda$-weighted relative Rota-Baxter operator.
\end{remark}

In the following, we characterize $\lambda$-weighted relative Rota-Baxter operators in terms of their graph.

\begin{prop}
A linear map $T : B \rightarrow A$ is a $\lambda$-weighted relative Rota-Baxter operator if and only if the graph $Gr (T) = \{ (T(u), u) | ~u \in B \}$ is a subalgebra of the $\lambda$-weighted semidirect product $A \ltimes_\lambda B$.
\end{prop}

\begin{proof}
For any $u, v \in B$, we have
\begin{align*}
(T(u), u) \bullet_\lambda (T(v), v) = \big( T(u) \cdot_A T(v),~ T(u) \cdot v + u \cdot T(v) + \lambda~ u \cdot_B v   \big).
\end{align*}
This is in $Gr (T)$ if and only if (\ref{rel-rb-id}) holds. Hence the result follows.
\end{proof}

As a consequence of the above proposition, we get the following.
\begin{prop}
Let $T : B \rightarrow A$ be a $\lambda$-weighted relative Rota-Baxter operator. Then $B$ carries a new associative algebra structure given by
\begin{align*}
u \cdot_T v := T(u) \cdot v + u \cdot T(v) + \lambda~ u \cdot_B v, ~ \text{ for } u, v \in B.
\end{align*}
\end{prop}

\medskip

Next, we consider $\lambda$-weighted modified associative Yang-Baxter equation (modified AYBE$_\lambda$) as an associative analogue of the modified Yang-Baxter equation considered in \cite{guo-lax}. We find its connection with $\lambda$-weighted Rota-Baxter operators.

\begin{defn}
Let $A$ be an associative algebra. For a linear map $R : A \rightarrow A$, the equation
\begin{align*}
R(a) \cdot_A R (b) = R \big( R(a) \cdot_A b + a \cdot_A R(b) \big) - \lambda^2 ~ a \cdot_A b, \text{ for } a, b \in A
\end{align*}
is called the  $\lambda$-weighted modified associative Yang-Baxter equation (modified AYBE$_\lambda$).
\end{defn}

\begin{prop}\label{rb-myb}
Let $A$ be an associative algebra. Then there is a one-to-one correspondence between solutions of modified AYBE$_\lambda$ and $\lambda$-weighted Rota-Baxter operators on $A$.
\end{prop}

\begin{proof}
Let $T$ be a $\lambda$-weighted Rota-Baxter operator on $A$. Take $R = \lambda \mathrm{id}_A + 2 T$. We observe that
\begin{align}\label{modified1}
R(a) \cdot_A R(b) = \lambda^2~ a \cdot_A b + 2 \lambda \big( T(a) \cdot_A b + a \cdot_A T(b) \big) + 4 ~T(a) \cdot_A T(b).
\end{align}
On the other hand, by a direct calculation
\begin{align}\label{modified2}
&R \big(   R(a) \cdot_A b + a \cdot_A R(b) \big) - \lambda^2 ~ a \cdot_A b \nonumber \\
&= \lambda^2~ a \cdot_A b + 2 \lambda \big( T(a) \cdot_A b + a \cdot_A T(b) \big)  + 4 ~T \big(  T(a) \cdot_A b + a \cdot_A T(b) + \lambda~ a \cdot_A b  \big).
\end{align}
Since $T$ is a $\lambda$-weighted Rota-Baxter operator on $A$, it follows from (\ref{modified1}) and (\ref{modified2}) that $R$ is a solution of modified AYBE$_\lambda$. Conversely, if $R$ is a solution of modified AYBE$_\lambda$, then it is easy to see that $T = \frac{1}{2} (R - \lambda \mathrm{id}_A)$ is a $\lambda$-weighted Rota-Baxter operator on $A$. This completes the proof.
\end{proof}

Let $A$ be an unital associative algebra with unit $1 \in A$. Let $r = r_{(1)} \otimes r_{(2)}$ be an element in $A\otimes A$. Here we use the Sweedler notation. We define three elements of $A \otimes A \otimes A$, namely,
\begin{align*}
r_{12} = r_{(1)} \otimes r_{(2)} \otimes 1, \qquad r_{13} = r_{(1)} \otimes 1 \otimes r_{(2)} ~~~~ \text{ and } ~~~~ r_{23} = 1 \otimes r_{(1)} \otimes r_{(2)}.
\end{align*} 

\begin{defn}
An element $r = r_{(1)} \otimes r_{(2)} \in A \otimes A$ is said to be a $\lambda$-weighted associative Yang-Baxter solution if the following holds
\begin{align}\label{w-aybe}
r_{13} r_{12} - r_{12} r_{23} + r_{23} r_{13} = \lambda r_{13}.
\end{align}
\end{defn}

We have the following relation between weighted associative Yang-Baxter solutions and weighted Rota-Baxter operators.

\begin{prop}
Let $A$ be an unital associative algebra and $r = r_{(1)} \otimes r_{(2)} \in A \otimes A$ be a $\lambda$-weighted associative Yang-Baxter solution. Then the map $T : A \rightarrow A$ defined by $T(a) = r_{(1)} \cdot_A a \cdot_A r_{(2)}$, is a $(-\lambda)$-weighted Rota-Baxter operator on $A$.
\end{prop}

\begin{proof}
Note that the identity (\ref{w-aybe}) can be written as
\begin{align*}
r_{(1)} \cdot_A \widetilde{r}_{(1)} \otimes \widetilde{r}_{(2)} \otimes r_{(2)} ~-~ r_{(1)} \otimes r_{(2)} \cdot_A \widetilde{r}_{(1)} \otimes \widetilde{r}_{(2)}~ +~ \widetilde{r}_{(1)} \otimes r_{(1)} \otimes r_{(2)} \cdot_A \widetilde{r}_{(2)} = \lambda~ r_{(1)} \otimes 1 \otimes r_{(2)},
\end{align*}
where $r_{(1)} \otimes r_{(2)}$ and $\widetilde{r}_{(1)} \otimes \widetilde{r}_{(2)}$ denote two copies of $r$. In the above identity, replacing the first tensor product by multiplication of $a$ and the second tensor product by multiplication of $b$, and using the definition of $R$, we get
\begin{align*}
T ( T(a) \cdot_A b ) - T(a) \cdot_A T(b) + T ( a \cdot_A T(b)) = \lambda T( a \cdot_A b).
\end{align*}
This shows that $T$ is a $(-\lambda)$-weighted Rota-Baxter operator on $A$.
\end{proof}

\medskip

Let $A$ and $B$ be two associative algebras and $B$ be an associative $A$-bimodule. Let $T, T' : B \rightarrow A$ be two $\lambda$-weighted relative Rota-Baxter operators. 

\begin{defn}
A morphism from $T$ to $T'$ consists of a pair $(\phi, \psi)$ of associative algebra morphisms $\phi : A \rightarrow A$ and $\psi : B \rightarrow B$ satisfying
\begin{align*}
\phi \circ T = T' \circ \psi , \quad \psi ( a \cdot u) = \phi (a) \cdot \psi (u)  ~~~ \text{ and } ~~~ \psi ( u \cdot a) = \psi (u) \cdot \phi (a), ~ \text{ for } a \in A, u \in B.
\end{align*}
\end{defn}

The set of all $\lambda$-weighted relative Rota-Baxter operators and morphisms between them forms a category, denoted by $\mathsf{rRB}_\lambda (B,A)$.

\begin{prop}
Let $(\phi, \psi)$ be a morphism of $\lambda$-weighted relative Rota-Baxter operators from $T$ to $T'$. Then $\psi : B \rightarrow B$ is a morphism of induced associative algebras from $(B, \cdot_T)$ to $(B, \cdot_{T'})$.
\end{prop}

\begin{proof}
For any $u, v \in B$, we have
\begin{align*}
\psi ( u \cdot_T v) =~& \psi \big( T(u) \cdot v + u \cdot T(v) + \lambda ~ u \cdot_B v \big) \\
=~& \phi (T(u)) \cdot \psi (v) + \psi (u) \cdot \phi (T(v)) + \lambda  ~ \psi(u) \cdot \psi (v)  \\
=~& T' (\psi (u)) \cdot \psi (v) + \psi (u ) \cdot T' (\psi (v)) + \lambda ~ \psi (u) \cdot_B \psi (v)  = \psi (u) \cdot_{T'} \psi (v).
\end{align*}
Hence the result follows.
\end{proof}

\section{Cohomology of $\lambda$-weighted relative Rota-Baxter operators}\label{sec-3}

The aim of this section is to provide Maurer-Cartan characterization of $\lambda$-weighted relative Rota-Baxter operators and define the cohomology of such operators.

\subsection{Maurer-Cartan characterization and Cohomology}\label{subsection31}

In this subsection, we construct a differential graded Lie algebra whose Maurer-Cartan elements are precisely $\lambda$-weighted relative Rota-Baxter operators. Using this characterization, we define the cohomology of a  $\lambda$-weighted relative Rota-Baxter operator.

We start by recalling the Gerstenhaber bracket \cite{gers2}. Let $V$ be a vector space and consider the graded space $\oplus_{n \geq 1} \mathrm{Hom} (V^{\otimes n}, V)$ of multilinear maps on $V$. It carries a degree $-1$ graded Lie bracket (called the Gerstenhaber bracket) given by
\begin{align*}
[f, g]_\mathsf{G} := \sum_{i=1}^m (-1)^{(i-1)(n-1)} ~f \circ_i g - (-1)^{(m-1)(n-1)} \sum_{i=1}^n (-1)^{(i-1)(m-1)}~ g \circ_i f,
\end{align*}
for $f \in \mathrm{Hom} (V^{\otimes m}, V)$ and $g \in \mathrm{Hom} (V^{\otimes n}, V)$, where
\begin{align*}
(f \circ_i g) ( v_1, \ldots, v_{m+n-1}) = f ( v_1, \ldots, v_{i-1}, g (v_i, \ldots, v_{i+n-1}), v_{i+n}, \ldots, v_{m+n-1}).
\end{align*}
In other words, $\big(  \oplus_{n \geq 0} \mathrm{Hom} (V^{\otimes n + 1}, V) , [~,~]_\mathsf{G} \big)$ is a graded Lie algebra. 
Note that a multiplication $\mu \in \mathrm{Hom} (V^{\otimes 2}, V)$ defines an associative structure on $V$ if and only if $[\mu, \mu]_\mathsf{G} = 0$.

\medskip

Let $A$ and $B$ be two associative algebras and $B$ be an associative $A$-bimodule. We use the notations $\mu_A$ and $\mu_B$ for associative  multiplications on $A$ and $B$, and  $l, r$ for left and right $A$-actions on $B$, respectively. Take $V = A \oplus B$ and consider the graded Lie algebra $\big( \oplus_{n \geq 0} \mathrm{Hom} (V^{\otimes n+1}, V), [~,~]_\mathsf{G} \big)$. Then it is easy to check that the graded subspace $\mathfrak{a} = \oplus_{n \geq 0} \mathrm{Hom} (B^{\otimes n+1}, A)$ is an abelian subalgebra. Moreover, we have the following observations.

\medskip

\medskip

\noindent {\bf Observation I.}

The element $\mu_A + l + r \in \mathrm{Hom}(V^{\otimes 2}, V)$ is a Maurer-Cartan element in the graded Lie algebra $\big( \oplus_{n \geq 0} \mathrm{Hom} (V^{\otimes n+1}, V), [~,~]_\mathsf{G} \big)$. Hence it induces a differential $d_{\mu_A + l + r} := [\mu_A + l + r, -]_\mathsf{G}$ on $\oplus_{n \geq 0} \mathrm{Hom} (V^{\otimes n+1}, V)$. Therefore, by the derived bracket construction \cite{voro}, the shifted space $\mathfrak{a} [-1] = \oplus_{n \geq 1} \mathrm{Hom} (B^{\otimes n}, A)$ 
carries a graded Lie bracket (called the derived bracket) given by
\begin{align*}
\llbracket P, Q \rrbracket := (-1)^m ~ [ d_{\mu_A + l + r}  (P), Q]_\mathsf{G} = (-1)^m~ [ ~[\mu_A + l + r, P]_\mathsf{G}, Q ]_\mathsf{G},
\end{align*}
for $P \in \mathrm{Hom}(B^{\otimes m}, A)$ and $Q \in \mathrm{Hom}(B^{\otimes n}, A)$. The explicit formula \cite{das-rota} is given by
\begin{align}\label{dgla-b}
&\llbracket P, Q \rrbracket (u_1, \ldots, u_{m+n}) \\
&=  \sum_{i=1}^m (-1)^{(i-1)n}~ P ( u_1, \ldots, u_{i-1}, Q (u_i, \ldots, u_{i+n-1}) \cdot u_{i+n}, \ldots, u_{m+n}) \nonumber \\
&- \sum_{i=1}^m (-1)^{in} ~ P (u_1, \ldots, u_{i-1}, u_i \cdot Q (u_{i+1}, \ldots, u_{i+n}), u_{i+n+1}, \ldots, u_{m+n})  \nonumber \\
&- (-1)^{mn} \bigg\{   \sum_{i=1}^n (-1)^{(i-1)m}~ Q ( u_1, \ldots, u_{i-1}, P (u_i, \ldots, u_{i+m-1}) \cdot u_{i+m}, \ldots, u_{m+n})      \nonumber \\
&- \sum_{i=1}^n (-1)^{im} ~ Q (u_1, \ldots, u_{i-1}, u_i \cdot P (u_{i+1}, \ldots, u_{i+m}), u_{i+m+1}, \ldots, u_{m+n})     \bigg\}  \nonumber  \\
& + (-1)^{mn} \big[   P(u_1, \ldots, u_m) \cdot_A Q (u_{m+1}, \ldots, u_{m+n}) - (-1)^{mn} ~ Q (u_1, \ldots, u_n) \cdot_A P (u_{n+1}, \ldots, u_{m+n}) \big].  \nonumber 
\end{align}
This graded Lie bracket can be extended to $\oplus_{n \geq 0} \mathrm{Hom} (B^{\otimes n}, A)$ by the following definitions:
\begin{align*}
&\llbracket P, a \rrbracket (u_1, \ldots, u_m) = \sum_{i=1}^m ~ P (u_1, \ldots, u_{i-1}, a \cdot u_i - u_i \cdot a, u_{i+1}, \ldots, u_m) \\
& \qquad \qquad \qquad \qquad \qquad + P (u_1, \ldots, u_m) \cdot_A a - a \cdot_A P (u_1, \ldots, u_m), \\
&\llbracket a, b \rrbracket = a \cdot_A b - b \cdot_A a, ~ \text{ for } P \in \mathrm{Hom} (B^{\otimes m}, A) \text{ and } a, b \in A. 
\end{align*}

\medskip

\medskip

\noindent {\bf Observation II.}

For any $\lambda \in {\bf k}$, the associative multiplication $- \lambda \mu_B$ can be considered as an element in $\mathrm{Hom}( V^{\otimes 2}, V)$. This is infact a Maurer-Cartan element in the graded Lie algebra $\big( \oplus_{n \geq 0} \mathrm{Hom} (V^{\otimes n+1}, V), [~,~]_\mathsf{G} \big)$. Therefore, $- \lambda \mu_B$ induces a differential $d_{ - \lambda \mu_B} := - [ \lambda \mu_B, -]_\mathsf{G}$ on $\oplus_{n \geq 1} \mathrm{Hom}(V^{\otimes n}, V)$. Moreover, the graded subspace $\oplus_{n \geq 1} \mathrm{Hom}(B^{\otimes n}, A)$ is closed under the differential $d_{- \lambda \mu_B}.$ We denote the restriction of the differential $d_{- \lambda \mu_B}$ to the subspace $\oplus_{n \geq 1} \mathrm{Hom}(B^{\otimes n}, A)$ by $d$, and it is given by
\begin{align}\label{dgla-d}
(df) (u_1, \ldots, u_{n+1}) = (-1)^{n-1} \sum_{i=1}^n (-1)^{i-1} ~f (u_1, \ldots, u_{i-1}, \lambda u_i \cdot_B u_{i+1}, u_{i+2}, \ldots, u_{n+1}).
\end{align}
The differential $d$ can be extended to $\oplus_{n \geq 0} \mathrm{Hom} (B^{\otimes n}, A)$ by $(da) (u) = T(u) \cdot_A a - a \cdot_A T(u)$, for $a \in A$ and $u \in B$.

\medskip

\medskip

\noindent {\bf Observation III.}

Finally, it is easy to see that the elements $\mu_A + l + r$ and $\lambda \mu_B$ satisfies the following compatibility
\begin{align*}
[\mu_A + l + r, \lambda \mu_B ]_\mathsf{G} = 0.
\end{align*}
Therefore, we have
\begin{align*}
&d \llbracket P, Q \rrbracket \\
&= (-1)^m ~ d~ [~[ \mu_A + l + r, P]_\mathsf{G}, Q ]_\mathsf{G} \\
&=  (-1)^{m-1} ~[ \lambda \mu_B, [~[ \mu_A + l + r, P]_\mathsf{G}, Q ]_\mathsf{G} ]_\mathsf{G} \\
&= (-1)^{m-1} ~ [~[ \lambda \mu_B, [\mu_A + l + r, P]_\mathsf{G} ]_\mathsf{G}, Q ]_\mathsf{G} + (-1)^{m-1} (-1)^m ~ [~[\mu_A + l + r, P]_\mathsf{G}, [\lambda \mu_B, Q ]_\mathsf{G} ]_\mathsf{G} \\
&= (-1)^{m-1} (-1)^1 ~[~[\mu_A + l + r, [\lambda \mu_B, P ]_\mathsf{G} ]_\mathsf{G}, Q]_\mathsf{G} + [~[\mu_A + l + r, P]_\mathsf{G}, dQ ]_\mathsf{G} \\
&= - (-1)^{m} ~[~[ \mu_A + l + r, dP]_\mathsf{G}, Q ]_\mathsf{G} + [~[\mu_A + l + r, P]_\mathsf{G}, dQ ]_\mathsf{G} \\
&= \llbracket dP, Q \rrbracket + (-1)^m ~ \llbracket P, dQ \rrbracket.
\end{align*}
This shows that $d$ is a graded derivation for the derived bracket $\llbracket ~, ~ \rrbracket$ on the graded space $\oplus_{n \geq 0} \mathrm{Hom}(B^{\otimes n }, A)$.

\medskip

\medskip

As a summary of the above three observations, we get that $\big( \oplus_{n \geq 0} \mathrm{Hom}(B^{\otimes n}, A), \llbracket ~, ~ \rrbracket, d  \big)$ is a differential graded Lie algebra. The importance of this differential graded Lie algebra is given by the following.

\begin{thm}
A linear map $T : B \rightarrow A$ is a $\lambda$-weighted relative Rota-Baxter operator if and only if $T$ is a Maurer-Cartan element in the differential graded Lie algebra $\big( \oplus_{n \geq 0} \mathrm{Hom}(B^{\otimes n}, A), \llbracket ~, ~ \rrbracket, d  \big)$.
\end{thm}

\begin{proof}
For a linear map $T: B \rightarrow A$, we have from (\ref{dgla-b}) and (\ref{dgla-d}) that
\begin{align*}
(d T + \frac{1}{2} \llbracket T, T \rrbracket ) (u, v) = T ( \lambda ~ u \cdot_B v ) + T (T(u) \cdot v + u \cdot T(v)) - T(u) \cdot_A T(v).
\end{align*}
This shows that $T$ satisfies $d T + \frac{1}{2} \llbracket T, T \rrbracket  = 0$ if and only if $T$ is a $\lambda$-weighted relative Rota-Baxter operator. Hence the proof.
\end{proof}

\medskip

Let $T : B \rightarrow A$ be a $\lambda$-weighted relative Rota-Baxter operator. Then $T$ induces a degree $1$ differential $d_T = d + \llbracket T, - \rrbracket$ on the graded space $\oplus_{n \geq 0} \mathrm{Hom}(B^{\otimes n}, A)$. The differential $d_T$ is explicitly given by
\begin{align}\label{t-diff}
&(d_T f) (u_1, \ldots, u_{n+1}) \\
&= T ( f (u_1, \ldots, u_n) \cdot u_{n+1} ) - (-1)^n ~ T (u_1 \cdot f (u_2, \ldots, u_{n+1})) \nonumber \\
&- (-1)^n \sum_{i=1}^n (-1)^{i-1} ~ f (u_1, \ldots, u_{i-1}, T(u_i) \cdot u_{i+1} + u_i \cdot T(u_{i+1}) + \lambda~ u_i \cdot_B u_{i+1}, \ldots, u_{n+1} ) \nonumber \\
&+ (-1)^n ~ T (u_1) \cdot_A f (u_2, \ldots, u_{n+1}) - f (u_1, \ldots, u_n) \cdot_A T (u_{n+1}), \nonumber
\end{align} 
for $f \in \mathrm{Hom}(B^{\otimes n}, A)$ and $u_1, \ldots, u_{n+1} \in B$. We define
\begin{align*}
C^n_T (B,A) := \mathrm{Hom}(B^{\otimes n}, A), ~ \text{ for } n \geq 0.
\end{align*}
Then $\{ C^\ast_T (B,A), d_T \}$ is a cochain complex. If $Z^n_T (B,A) = \{ f \in C^n_T (B,A) |~ d_T f = 0 \}$ is the space of $n$-cocycles and $B^n_T (B,A) = \{ d_T f |~ f \in C^{n-1}_T (B,A) \}$ is the space of $n$-coboundaries, then we have $B^n_T (B, A) \subset Z^n_T (B,A)$, for $n \geq 0$. The corresponding quotients
\begin{align*}
H^n_T (B,A) := \frac{Z^n_T (B,A)  }{ B^n_T (B,A) }, ~\text{ for } n \geq 0
\end{align*}
are called the cohomology groups of $T$.

\medskip

\medskip

Note that the differential $d_T$ makes the triple $\big( \oplus_{n \geq 0} \mathrm{Hom}(B^{\otimes n}, A), \llbracket ~, ~ \rrbracket, d_T  \big)$ into a new differential graded Lie algebra. This new structure governs Maurer-Cartan deformations of $T$ as described by the following.

\begin{thm}
Let $T: B \rightarrow A$ be a $\lambda$-weighted relative Rota-Baxter operator. For any linear map $T' : B \rightarrow A$, the sum $T + T'$ is a $\lambda$-weighted relative Rota-Baxter operator if and only if $T'$ is a Maurer-Cartan element in the differential graded Lie algebra  $\big( \oplus_{n \geq 0} \mathrm{Hom}(B^{\otimes n}, A), \llbracket ~, ~ \rrbracket, d_T  \big)$.
\end{thm}

\begin{proof}
We observe that
\begin{align*}
d (T + T') + \frac{1}{2} \llbracket T + T', T + T' \rrbracket 
&= dT + dT' + \frac{1}{2} \big( \llbracket T, T \rrbracket + 2 \llbracket T, T' \rrbracket + \llbracket T', T' \rrbracket \big) \\
&= dT' + \llbracket T, T' \rrbracket + \frac{1}{2} \llbracket T', T' \rrbracket \\
&= d_T (T') + \frac{1}{2} \llbracket T', T' \rrbracket.
\end{align*}
Hence the result follows.
\end{proof}

\subsection{A new interpretation of the cohomology}\label{subsec-hoch}

In this subsection, we show that the cohomology of a $\lambda$-weighted relative Rota-Baxter operator $T$ defined above can be expressed as the Hochschild cohomology of a suitable associative algebra with coefficients in an appropriate bimodule. We start with the following result.

\begin{prop}
Let $T: B \rightarrow A$ be a $\lambda$-weighted relative Rota-Baxter operator. Then $A$ is a bimodule over the induced associative algebra $(B, \cdot_T)$ with left and right actions given by
\begin{align*}
l^T_u ( a) := T(u) \cdot_A a - T( u \cdot a) ~~~~ \text{ and } ~~~~ r^T_u (a) := a \cdot_A T(u) - T(a \cdot u), ~ \text{ for } u \in B, a \in A. 
\end{align*}
\end{prop}

\begin{proof}
For any $u, v \in B$ and $a \in A$, we have
\begin{align*}
&l^T_{u \cdot_T v} (a) - l^T_u l^T_v (a) \\
&= T (u \cdot_T v ) \cdot_A a - T ((u \cdot_T v) \cdot a) - l^T_u ( T(v) \cdot_A a - T (v \cdot a)) \\
&= (T(u) \cdot_A T(v) ) \cdot_A a - T \big(  (T(u) \cdot v + u \cdot T(v) + \lambda~ u \cdot_B v) \cdot a  \big) \\
& \qquad \qquad \qquad - T(u) \cdot_A (T(v) \cdot_A a - T (v \cdot a)) + T \big( u \cdot (T(v) \cdot_A a - T (v \cdot a))  \big) \\
&= - T \big( T(u) \cdot ( v \cdot a) + u \cdot (T(v) \cdot_A a) + \lambda ~ u \cdot_B (v \cdot a) \big) \\
& \qquad \qquad \qquad + T(u) \cdot_A T (v \cdot a) + T (u \cdot (T(v) \cdot_A a)) - T (u \cdot T (v \cdot a)) \\
&= - T \big( T(u) \cdot (v \cdot a) + \lambda ~u  \cdot_B (v \cdot a)  \big) + T \big(   T(u) \cdot (v \cdot a) + u \cdot T(v \cdot a) + \lambda ~ u \cdot_B (v \cdot a) \big)  - T (u \cdot T (v \cdot a) ) \\
&= 0. 
\end{align*}
Similarly, 
\begin{align*}
&r_v^T l_u^T (a) - l_u^T r_v^T (a) \\
&= r_v^T (T(u) \cdot_A a - T(u \cdot a)) - l_u^T (a \cdot_A T(v) - T(a \cdot v)) \\
&= (T(u) \cdot_A a - T(u \cdot a)) \cdot_A T(v)  - T \big(  (T(u) \cdot_A a - T(u \cdot a)) \cdot v \big) \\
& \qquad \qquad \qquad - T(u) \cdot_A (a \cdot_A T(v) - T (a \cdot v)) + T \big( u \cdot (a \cdot_A T(v) - T( a \cdot v) )  \big) \\
&= - T( u \cdot a) \cdot_A T(v) - T (T(u) \cdot (a \cdot v) - T(u \cdot a) \cdot v) + T (u) \cdot_A T (a \cdot v) + T \big( u \cdot (a \cdot_A T(v)) - u \cdot T (a \cdot v)  \big) \\
&= - T \big(  T(u \cdot a) \cdot v + (u \cdot a) \cdot T(v) + \lambda ~ (u \cdot a) \cdot_B v  \big) - T (T(u) \cdot ( a \cdot v)   - T (u \cdot a) \cdot v) \\
& \qquad \qquad \qquad + T \big(  T(u) \cdot (a \cdot v) + u \cdot T ( a \cdot v) + \lambda~ u \cdot_B (a \cdot v) \big) + T((u \cdot a) \cdot T(v) - u \cdot T(a \cdot v)) \\
&= 0
\end{align*}
and 
\begin{align*}
&r_v^T r_u^T (a) - r^T_{u \cdot_T v} (a) \\
&= r^T_v (a \cdot_A T(u) - T(a \cdot u)) - a \cdot_A T (u \cdot_T v) + T (a \cdot (u \cdot_T v)) \\
&= (a \cdot_A T(u) - T (a \cdot u)) \cdot_A T(v) - T \big(  (a \cdot_A T(u) - T (a \cdot u)) \cdot v   \big) \\
& \qquad \qquad \qquad - a \cdot_A (T(u) \cdot_A T(v)) + T (a \cdot (T(u) \cdot v + u \cdot T(v) + \lambda~ u \cdot_B v) ) \\
&= -T (a \cdot u) \cdot_A T(v) - T (a \cdot (T(u) \cdot v) - T(a \cdot u) \cdot v) \\
& \qquad \qquad \qquad + T (a \cdot (T(u) \cdot v)) + T (a \cdot (u \cdot T(v))) + \lambda~ T(  a \cdot (u \cdot_B v))   \\
&= - T \big( T( a \cdot u) \cdot v + ( a \cdot u) \cdot T(v) + \lambda ~(a \cdot u) \cdot_B v   \big) + T (T(a \cdot u) \cdot v + T ((a \cdot u) \cdot T(v)) + \lambda ~ T (( a \cdot u) \cdot_B v )\\
&= 0.
\end{align*}
This proves the result.
\end{proof}

\medskip

Let $T : B \rightarrow A$ be a $\lambda$-weighted relative Rota-Baxter operator. Then it follows from the previous proposition that one may consider the Hochschild complex $\{ C^\ast_\mathsf{H} (B,A) , d_\mathsf{H} \}$ of the associative algebra $(B, \cdot_T)$ with coefficients in the bimodule $(A, l^T, r^T).$ More precisely, the $n$-th cochain group 
\begin{align*}
C^n_\mathsf{H} (B,A) := \mathrm{Hom}(B^{\otimes n}, A), ~ \text{ for } n \geq 0
\end{align*} 
and the coboundary map $d_\mathsf{H} : C^n_\mathsf{H} (B,A) \rightarrow C^{n+1}_\mathsf{H} (B,A)$ given by
\begin{align}\label{hoch-diff}
(d_\mathsf{H} f ) (u_1, \ldots, u_{n+1}) =~& T(u_1) \cdot_A f (u_2, \ldots, u_{n+1}) - T (u_1 \cdot f (u_2, \ldots, u_{n+1})) \\
+~& \sum_{i=1}^n (-1)^i ~ f (u_1, \ldots, u_{i-1}, T(u_i) \cdot u_{i+1} + u_i \cdot T (u_{i+1}) + \lambda ~ u_i \cdot_B u_{i+1}, \ldots, u_{n+1} ) \nonumber \\
+~&(-1)^{n+1} ~ f(u_1, \ldots, u_n) \cdot_A T(u_{n+1}) - (-1)^{n+1} ~ T (f(u_1, \ldots, u_n) \cdot u_{n+1} ), \nonumber
\end{align}
for $f \in C^n_\mathsf{H}(B,A)$ and $u_1, \ldots, u_{n+1} \in B$. We denote the corresponding cohomology groups by $H^\ast_\mathsf{H}(B,A)$.

It follows from the expressions of (\ref{t-diff}) and (\ref{hoch-diff}) that $d_T f = (-1)^n ~d_\mathsf{H} f$, for $f \in \mathrm{Hom}(B^{\otimes n}, A)$. In other words, the differentials $d_T$ and $d_\mathsf{H}$ are same up to a sign, which implies that
\begin{align*}
H^\ast_T (B,A) \cong H^\ast_\mathsf{H} (B,A).
\end{align*}

\subsection{Particular case: $\lambda$-weighted Rota-Baxter operators}

Let $A$ be an associative algebra. Then $A$ is itself an associative $A$-bimodule, called the adjoint $A$-bimodule. Moreover, we have seen in Remark \ref{rel-not} that a $\lambda$-weighted Rota-Baxter operator on $A$ is a particular case of $\lambda$-weighted relative Rota-Baxter operator. Therefore, one may adopt the previous results in this particular case.

We summarize the results of Subsection \ref{subsection31} in the following theorem.

\begin{thm}
Let $A$ be an associative algebra and $\lambda \in {\bf k}$ be a fixed scalar. Then
\begin{itemize}
\item[(i)] there is a differential graded Lie algebra $\big( \oplus_{n \geq 0} \mathrm{Hom}(A^{\otimes n}, A), \llbracket ~, ~ \rrbracket, d   \big)$ on the graded space of multilinear maps on $A$, where the bracket $\llbracket ~, ~ \rrbracket$ and the differential $d$ are given by the formulas (\ref{dgla-b}) and (\ref{dgla-d}), respectively.
\item[(ii)] A linear map $T: A \rightarrow A$ is a $\lambda$-weighted Rota-Baxter operator on $A$ if and only if $T \in \mathrm{Hom}(A, A)$ is a Maurer-Cartan element in the above differential graded Lie algebra.
\item[(iii)] Let $T$ be a $\lambda$-weighted Rota-Baxter operator on $A$. For any linear map $T' : A \rightarrow A$, the sum $T + T'$ is a $\lambda$-weighted Rota-Baxter operator on $A$ if and only if $T'$ is a Maurer-Cartan element in the differential graded Lie algebra $\big( \oplus_{n \geq 0} \mathrm{Hom}(A^{\otimes n}, A), \llbracket ~, ~ \rrbracket, d_T = d + \llbracket T, - \rrbracket   \big)$.
\end{itemize}
\end{thm}

It follows that a $\lambda$-weighted Rota-Baxter operator $T$ induces a cochain complex $\{ C^\ast_T (A,A), d_T \}$. The corresponding cohomology groups are called the cohomology of $T$.

\medskip

On the other hand, the $\lambda$-weighted Rota-Baxter operator $T$ induces a new associative structure on $A$ with the product given by
\begin{align*}
a \cdot_T b = T(a) \cdot_A b + a \cdot_A T(b) + \lambda ~ a \cdot_A b, ~ \text{ for } a, b \in A.
\end{align*}
The vector space $A$ can be given a bimodule structure over the associative algebra $(A, \cdot_T)$ with left and right actions
\begin{align*}
l_a^T (b) = T(a) \cdot_A b - T (a \cdot_A b) ~~~~ \text{ and } ~~~~ r_a^T (b) = b \cdot_A T(a) - T ( b \cdot_A a), ~ \text{ for } a, b \in A.
\end{align*}
Moreover, the cohomology of $T$ is isomorphic to the Hochschild cohomology of $(A, \cdot_T)$ with coefficients in the above bimodule $(A, l^T, r^T)$.

\section{Deformations of $\lambda$-weighted relative Rota-Baxter operators}\label{sec-4}
In this section, we study deformations of a $\lambda$-weighted relative Rota-Baxter operator $T$ from cohomological points of view. We introduce Nijenhuis elements associated to $T$ that generate trivial linear deformations of $T$. We also find a sufficient condition for the rigidity of $T$. Finally, given a finite order deformation of $T$, we construct a second cohomology class in the cohomology of $T$, called the obstruction class. The vanishing of the obstruction class ensures that the deformation is extensible. 

\subsection{Linear deformations and Nijenhuis elements}

Let $A$ and $B$ be two associative algebras and $B$ be an associative $A$-bimodule. Let $T : B \rightarrow A$ be a $\lambda$-weighted relative Rota-Baxter operator.

\begin{defn}
A linear map $T_1 : B \rightarrow A$ is said to generate a linear deformation of $T$ if the linear sum $T_t := T + t T_1$ is a $\lambda$-weighted relative Rota-Baxter operator for all values of $t$. In this case, we say that $T_t$ is a linear deformation of $T$.
\end{defn}

Note that $T_1$ generates a linear deformation of $T$ if and only if the followings are hold
\begin{align}
T(u) \cdot_A T_1 (v) + T_1(u) \cdot_A T(v) =~& T (T_1(u) \cdot v + u \cdot T_1 (v)) + T_1 (T(u) \cdot v + u \cdot T(v) + \lambda u \cdot_B v), \label{lin1}\\
T_1 (u) \cdot_A T_1 (v) =~& T_1 (T_1 (u) \cdot v + u \cdot T_1(v)), \label{lin2}
\end{align}
for all $u, v \in B$.  It follows from (\ref{lin1}) that $d_T (T_1) = 0$, i.e., $T_1$ is a $1$-cocycle in the cohomology complex of $T$. On the other hand, the identity (\ref{lin2}) implies that $T_1$ is a relative Rota-Baxter operator (of weight $0$).

\begin{defn}
Two linear deformations $T_t = T+ t T_1$ and $T_t' = T + t T_1'$ of a $\lambda$-weighted relative Rota-Baxter operator $T$ are said to be equivalent if there exists an element $a_0 \in A$ such that $(\phi = \mathrm{id}_A + t (l_{a_0}^{\mathrm{ad}} - r_{a_0}^{\mathrm{id}}) , ~ \psi = \mathrm{id}_B + t (l_{a_0} - r_{a_0}))$ is a morphism of $\lambda$-weighted relative Rota-Baxter operators from $T_t$ to $T_t'$.
\end{defn}

The condition in the above definition is equivalent to the followings (see \cite{das-rota} for similar observation)
\begin{align}
(a_0 \cdot_A a - a \cdot_A a_0) \cdot_A (a_0 \cdot_A b - b \cdot_A a_0) = 0 ~~~ \text{ and } ~~~ (a_0 \cdot u - u \cdot a_0) \cdot_B (a_0 \cdot v - v \cdot a_0) = 0, \label{equiv1}\\
\qquad \qquad \qquad \begin{cases}
T_1 (u ) - T_1' (u) = T (a_0 \cdot u - u \cdot a_0) - (a_0 \cdot_A T(u) - T(u) \cdot_A a_0 ), \\
a_0 \cdot_A T_1(u) - T_1 (u) \cdot_A a_0 = T_1' (a_0 \cdot u - u \cdot a_0),
\end{cases} \label{equiv2}\\
l_{(a_0 \cdot_A a - a \cdot_A a_0)} l_{a_0} = l_{(a_0 \cdot_A a - a \cdot_A a_0)} r_{a_0}   ~~~~ \text{ and } ~~~~ r_{(a_0 \cdot_A a - a \cdot_A a_0)} l_{a_0} = r_{(a_0 \cdot_A a - a \cdot_A a_0)} r_{a_0}, \label{equiv3}
\end{align}
for all $a, b \in A$ and $u, v \in B$. It follows from (\ref{equiv2}) that $(T_1 - T_1') (u) = d_T (a_0)(u)$, for $u \in B$. Thus, we have the following.

\begin{prop}
Let $T_t = T+ t T_1$ and $T_t' = T + tT_1'$ be two equivalent linear deformations of $T$. Then $T_1$ and $T_1'$ corresponds to the same cohomology class in $H^1_T (B,A).$
\end{prop}

We now introduce Nijenhuis elements associated to $T$. This generalizes the similar notion defined for weight $0$ relative Rota-Baxter operators \cite{das-rota}.

\begin{defn}
An element $a_0 \in A$ is called a Nijenhuis element associated to a $\lambda$-weighted relative Rota-Baxter $T$ if $a_0$ satisfies (\ref{equiv1}), (\ref{equiv3}) and
\begin{align*}
a_0 \cdot_A (l_u^T (a_0) - r_u (a_0) ) - (l_u^T (a_0) - r_u^T (a_0) ) \cdot_A a_0 = 0, \text{ for all } u \in B. 
\end{align*}

We denote the set of all Nijenhuis elements by $\mathrm{Nij}(T)$. A linear deformation $T_t$ is said to be trivial if $T_t$ is equivalent to $T_t' = T$. It follows from the above discussion that a trivial linear deformation gives rise to a Nijenhuis element. It was shown in \cite{das-rota} that Nijenhuis elements also generate trivial linear deformations of relative Rota-Baxter operator (of weight $0$). Similarly, we can prove the same result for a weighted relative Rota-Baxter operator. We only state the result and refer to \cite{das-rota} for the proof.

\begin{prop}
Let $T: B \rightarrow A$ be a $\lambda$-weighted relative Rota-Baxter operator. For any $a_0 \in \mathrm{Nij}(T)$, the linear map $T_1 = d_T(a) : B \rightarrow A$ generates a trivial linear deformation of $T$.
\end{prop}

\end{defn}
\subsection{Formal deformations}
In this subsection, we study formal deformations of a $\lambda$-weighted relative Rota-Baxter operator $T$ by keeping the underlying algebras $A$ and $B$ intact.

Let $A$ and $B$ be two associative algebras and $B$ be an associative $A$-bimodule. Consider the space $A[[t]]$ of formal power series in $t$ with coefficients from $A$. The associative structure on $A$ induces an associative product on $A[[t]]$ by ${\bf k}[[t]]$-bilinearity. Similarly, $B[[t]]$ is an associative algebra over ${\bf k}[[t]]$, and $B[[t]]$ is an associative $A[[t]]$-bimodule.

\begin{defn}
Let $T: B \rightarrow A$ be a $\lambda$-weighted relative Rota-Baxter operator. A formal deformation of $T$ is given by a formal sum
\begin{align*}
T_t = T_0 + t T_1 + t^2 T_2 + \cdots  ~\in \mathrm{Hom}(B,A)[[t]] ~~ \text{ with } T_0 = T
\end{align*}
such that the ${\bf k}[[t]]$-linear map $T_t : B[[t]] \rightarrow A[[t]]$ is a $\lambda$-weighted relative Rota-Baxter operator.
\end{defn}

Note that $T_t$ is a formal deformation of $T$ if and only if
\begin{align}\label{def-eqn-sys}
\sum_{i+j = n} T_i (u) \cdot_A T_j (v) = \sum_{i+j = n} T_i \big(     T_j (u ) \cdot v + u \cdot T_j (v) \big) + \lambda ~ T_n (u \cdot_B v),~ \text{ for } n=0, 1, \ldots .
\end{align}
The identity (\ref{def-eqn-sys}) holds for $n=0$ as $T_0 = T$ is a $\lambda$-weighted relative Rota-Baxter operator. For $n =1$, we get
\begin{align*}
T(u) \cdot_A T_1 (v) + T_1 (u) \cdot_A T(v) = T ( T_1 (u) \cdot v + u \cdot T_1 (v)) ~+~ T_1 ( T (u) \cdot v + u \cdot T (v) +  \lambda ~ u \cdot_B v),
\end{align*}
for $u, v \in B$. This implies that $T_1$ is a $1$-cocycle in the cohomology complex of $T$. This is called the infinitesimal of the deformation $T_t$. More generally, if $T_1 = \cdots = T_{k-1} = 0$ and $T_k$ is the first nonvanishing term after $T_0 = T$ in the expression of $T_t$, then $T_k$ is a $1$-cocycle in the cohomology complex of $T$.

\begin{defn}
Two formal deformations $T_t$ and $T_t'$ of a $\lambda$-weighted relative Rota-Baxter operator $T$ are said to be equivalent if there exists an element $a_0 \in A$, linear maps $\phi_i \in \mathrm{Hom}(A,A)$ and $\psi_i \in \mathrm{Hom}(B, B)$ for $i \geq 2$, such that
\begin{align*}
\big( \phi_t= \mathrm{id}_A + t (l_{a_0}^{\mathrm{ad}} - r_{a_0}^{\mathrm{ad}}) + \sum_{i \geq 2} t^i \phi_i, ~ \psi_{t} = \mathrm{id}_B + t (l_{a_0} - r_{a_0}) + \sum_{i \geq 2} t^i \psi_i  \big)
\end{align*}
is a morphism of $\lambda$-weighted relative Rota-Baxter operators from $T_t$ to $T_t'$.
\end{defn}

The following result is similar to the case of linear deformations. Hence we omit the details.

\begin{prop}
Let $T_t$ and $T_t'$ be two equivalent formal deformations of a $\lambda$-weighted relative Rota-Baxter operator $T$. Then the corresponding infinitesimals are cohomologous, hence corresponds to the same cohomology class in $H^1_T (B,A).$
\end{prop}

\medskip

\begin{defn}
A $\lambda$-weighted relative Rota-Baxter operator $T$ is said to be rigid if any deformation $T_t$ is equivalent to the undeformed one $T_t' = T$.
\end{defn}

One may find the following sufficient condition for the rigidity of $T$ in terms of Nijenhuis elements. The proof is similar to \cite[Proposition 4.16]{das-rota}. Hence we will not repeat it here.

\begin{thm}
Let $T$ be a $\lambda$-weighted relative Rota-Baxter operator. If $Z^1_T (B,A) = d_T (\mathrm{Nij}(T))$ then $T$ is rigid.
\end{thm}

\subsection{Finite order deformations}
Let $A$ and $B$ be two associative algebras and $B$ be an associative $A$-bimodule. For any $N \geq 1$, consider the space $A[[t]]/(t^{N+1})$ which inherits an associative algebra structure over the base ring ${\bf k}[[t]]/(t^{N+1})$. Similarly, the space $B[[t]]/(t^{N+1})$  is an associative algebra over  ${\bf k}[[t]]/(t^{N+1})$, and an associative $A[[t]]/(t^{N+1})$-bimodule.

\begin{defn}
Let $T: B \rightarrow A$ be a $\lambda$-weighted relative Rota-Baxter operator. An order $N$ deformation of $T$ consists of a polynomial of the form $T_t = \sum_{i=0}^N t^i T_i$ with $T_0 = T$, such that the  ${\bf k}[[t]]/(t^{N+1})$-linear map $T_t :  B[[t]]/(t^{N+1}) \rightarrow A[[t]]/(t^{N+1})$ is a $\lambda$-weighted relative Rota-Baxter operator.
\end{defn}

Note that $T_t = \sum_{i=0}^N t^i T_i$ is an order $N$ deformation of $T$ if and only if the identities (\ref{def-eqn-sys}) are hold for $n =0, 1, \ldots, N$. They can be equivalently expressed as
\begin{align}\label{fin-eq}
d_T (T_n ) = - \frac{1}{2} \sum_{i+j = n; i, j \geq 1}\llbracket T_i, T_j \rrbracket,~ \text{ for } n =0, 1, \ldots, N.
\end{align}

\begin{defn}
An order $N$ deformation $T_t = \sum_{i=0}^N t^i T_i$ of a $\lambda$-weighted relative Rota-Baxter operator $T$ is said to extensible if there exists a linear map $T_{N+1} : B \rightarrow A$ such that
\begin{align*}
\widetilde{T_t} = T_t + t^{N+1} T_{N+1} = \sum_{i=0}^{N+1} t^i T_i
\end{align*}
is a deformation of order $N+1$.
\end{defn}

Thus, in the case of extension, one more equation need to be satisfied, namely,
\begin{align}\label{exten-n}
d_T (T_{N+1}) = - \frac{1}{2} \sum_{i+j = N+1; i, j \geq 1} \llbracket T_i, T_j \rrbracket.
\end{align}
Observe that the right-hand side of the above equation does not contain the term $T_{N+1}$. Hence it depends only on the given order $N$ deformation $T_t$. Moreover, it is a $2$-cochain in the cohomology complex of $T$. We denote this $2$-cochain by $Ob_{T_t}.$

\begin{prop}
The cochain $Ob_{T_t}$ is a $2$-cocycle in the cohomology complex of $T$.
\end{prop}

\begin{proof}
We observe that
\begin{align*}
&d_T ( - \frac{1}{2} \sum_{i+j = N+1, i, j \geq 1} \llbracket T_i, T_j \rrbracket ) \\
&=  - \frac{1}{2} \sum_{i+j = N+1, i, j \geq 1}   ( d \llbracket T_i, T_j \rrbracket + \llbracket T, \llbracket T_i, T_j \rrbracket \rrbracket ) \\
&= - \frac{1}{2} \sum_{i+j = N+1, i, j \geq 1}    \big( \llbracket d T_i, T_j \rrbracket -  \llbracket T_i, dT_j \rrbracket  + \llbracket \llbracket T, T_i \rrbracket, T_j \rrbracket  - \llbracket T_i, \llbracket T, T_j \rrbracket \rrbracket \big) \\
&=  - \frac{1}{2} \sum_{i+j = N+1, i, j \geq 1}    \big( \llbracket d_T T_i, T_j \rrbracket -  \llbracket T_i, d_T T_j \rrbracket \big)  \\
&= \frac{1}{4} \sum_{i_1 +  i_2 + j = N+1, i_1, i_2, j \geq 1} \llbracket  \llbracket T_{i_1}, T_{i_2} \rrbracket, T_j \rrbracket - \frac{1}{4} \sum_{i + j_1 + j_2 = N+1, i, j_1, j_2 \geq 1} \llbracket T_i, \llbracket T_{j_1}, T_{j_2} \rrbracket \rrbracket  ~~~ (\text{from } (\ref{fin-eq}))\\
&= \frac{1}{2} \sum_{i + j + k = N+1, i, j, k \geq 1} \llbracket  \llbracket T_{i}, T_{j} \rrbracket, T_k \rrbracket = 0.
\end{align*}
This completes the proof.
\end{proof}

Thus, given an order $N$ deformation $T_t$, we obtain a second cohomology class $[Ob_{T_t} ] \in H^2_T (B,A)$ in the cohomology of $T$. This is called the obstruction class to extend the deformation $T_t$. Moreover, from the identity (\ref{exten-n}), we get the following.

\begin{thm}
An order $N$ deformation $T_t$ of a $\lambda$-weighted relative Rota-Baxter operator $T$ is extensible if and only if the corresponding obstruction class $[Ob_{T_t} ] \in H^2_T (B,A)$ is trivial.
\end{thm}

\begin{corollary}
\begin{itemize}
\item[(i)] If $H^2_T (B,A) = 0$ then any finite order deformation of $T$ is extensible.

\item[(ii)] If $H^2_T (B,A) = 0$ then every $1$-cocycle in the cohomology complex of $T$ is the infinitesimal of some formal deformation of $T$.
\end{itemize}
\end{corollary}

\begin{remark}
One may also study various aspects of deformations of $\lambda$-weighted Rota-Baxter operators on an algebra $A$. Since we will get similar results as above, we do not repeat them here.
\end{remark}

\section{The Lie case}\label{sec-5}
In this section, we define the cohomology of $\lambda$-weighted relative Rota-Baxter operators in the context of Lie algebras. We find its relation with the cohomology of $\lambda$-weighted relative Rota-Baxter operators on associative algebras.

\subsection{Weighted relative Rota-Baxter operators and their Cohomology}
Let $(\mathfrak{g}, [~,~]_\mathfrak{g})$ and $(\mathfrak{h}, [~,~]_\mathfrak{h})$ be two Lie algebras. Suppose the Lie algebra $\mathfrak{g}$ acts on $\mathfrak{h}$ by a Lie algebra homomorphism $\rho : \mathfrak{g} \rightarrow \mathrm{Der}(\mathfrak{h})$. In this case, we also say that $\mathfrak{h}$ is a Lie $\mathfrak{g}$-module.

\begin{defn}
\begin{itemize}
\item[(i)] Let $\mathfrak{g}$ and $\mathfrak{h}$ be two Lie algebras, and let $\mathfrak{g}$ acts on $\mathfrak{h}$ by a Lie algebra homomorphism $\rho : \mathfrak{g} \rightarrow \mathrm{Der}(\mathfrak{h})$. A linear map $T : \mathfrak{h} \rightarrow \mathfrak{g}$ is said to be a $\lambda$-weighted relative Rota-Baxter operator on $\mathfrak{h}$ over the Lie algebra $\mathfrak{g}$ if
\begin{align*}
[T(u), T(v)]_\mathfrak{g} = T \big(   \rho (Tu) v - \rho (Tv) u + \lambda [u, v]_\mathfrak{h} \big), ~ \text{ for } u, v \in \mathfrak{h}.
\end{align*}
For simplicity, we say that $T$ is a  $\lambda$-weighted relative Rota-Baxter operator.
\item[(ii)] Let $T, T' : \mathfrak{h} \rightarrow \mathfrak{g}$ be two $\lambda$-weighted relative Rota-Baxter operators. A morphism from $T$ to $T'$ is a pair $(\phi, \psi)$ consist of Lie algebra homomorphisms $\phi : \mathfrak{g} \rightarrow \mathfrak{g}$ and $\psi : \mathfrak{h} \rightarrow \mathfrak{h}$ satisfying additionally
\begin{align*}
\phi \circ T = T' \circ \psi ~~~~ \text{ and } \psi (\rho (x) u) = \rho (\phi (x)) \psi (u), ~ \text{ for } x \in \mathfrak{g}, u \in \mathfrak{h}.
\end{align*}
\end{itemize}
The set of all $\lambda$-weighted relative Rota-Baxter operators and morphisms between them forms a category. We denote this category by $\mathsf{rRB}_\lambda (\mathfrak{h}, \mathfrak{g}).$
\end{defn}

\medskip

In the following, we introduce $\lambda$-weighted modified Yang-Baxter equation in a Lie algebra and find a connection with $\lambda$-weighted Rota-Baxter operators.

\begin{defn}
Let $\mathfrak{g}$ be a Lie algebra. For a linear map $R : \mathfrak{g} \rightarrow \mathfrak{g}$, the equation
\begin{align*}
[R(x), R(y)]_\mathfrak{g} = R \big( [R(x), y]_\mathfrak{g} + [x, R(y)]_\mathfrak{g} \big) - \lambda^2 [x,y]_\mathfrak{g},~ \text{ for } x, y \in \mathfrak{g}
\end{align*}
is called the $\lambda$-weighted modified Yang-Baxter equation (modified YBE$_\lambda$).
\end{defn}

The following result is similar to Proposition \ref{rb-myb}. Hence we will omit the details.

\begin{prop}
Let $\mathfrak{g}$ be a Lie algebra. A linear map $T : \mathfrak{g} \rightarrow \mathfrak{g}$ is a $\lambda$-weighted Rota-Baxter operator if and only if $R = \lambda \mathrm{id}_\mathfrak{g} + 2T$ is a solution of modified YBE$_\lambda$.
\end{prop}

Let $\mathfrak{g}$ acts on $\mathfrak{h}$ by a Lie algebra homomorphism $\rho : \mathfrak{g} \rightarrow \mathrm{Der}(\mathfrak{h})$. Then for any $\lambda \in {\bf k}$, the direct sum $\mathfrak{g} \oplus \mathfrak{h}$ carries a Lie bracket given by
\begin{align*}
[(x,u), (u, v)]_{\ltimes_\lambda} := \big(  [x,y]_\mathfrak{g},~ \rho (x) u - \rho(y) v + \lambda [u, v]_\mathfrak{h} \big), \text{ for } x, y \in \mathfrak{g} \text{ and } u, v \in \mathfrak{h}. 
\end{align*}
This is called the $\lambda$-weighted semidirect product. With this notation, we have the following characterization of $\lambda$-weighted relative Rota-Baxter operators.

\begin{prop}
A linear map $T : \mathfrak{h} \rightarrow \mathfrak{g}$ is a $\lambda$-weighted relative Rota-Baxter operator if and only if the graph $Gr (T) = \{ (T(u), u) |~ u \in \mathfrak{h} \}$ is a subalgebra of the $\lambda$-weighted semidirect product Lie algebra $(\mathfrak{g} \oplus \mathfrak{h}, [~,~]_{\ltimes_\lambda})$.
\end{prop}

As a consequence, we get the following.

\begin{prop}
Let $T: \mathfrak{h} \rightarrow \mathfrak{g}$ be a $\lambda$-weighted relative Rota-Baxter operator. Then $\mathfrak{h}$ carries a new Lie algebra structure with bracket
\begin{align*}
[u, v]_T := \rho (Tu) v - \rho (Tv) u + \lambda [u,v]_\mathfrak{h}, ~ \text{ for } u, v \in \mathfrak{h}.
\end{align*}
\end{prop}

\medskip

We will now define cohomology of a $\lambda$-weighted relative Rota-Baxter operator.  Let $\mathfrak{g}$ and $\mathfrak{h}$ be two Lie algebras and let $\mathfrak{g}$ acts on $\mathfrak{h}$ by a Lie algebra homomorphism $\rho : \mathfrak{g} \rightarrow \mathrm{Der}(\mathfrak{h})$. Let $\pi_\mathfrak{g} \in \mathrm{Hom}(\wedge^2 \mathfrak{g}, \mathfrak{g})$ and $\pi_\mathfrak{h} \in \mathrm{Hom}(\wedge^2 \mathfrak{h}, \mathfrak{h})$ be the elements corresponding to the Lie brackets of $\mathfrak{g}$ and $\mathfrak{h}$, respectively. Let $W = \mathfrak{g} \oplus \mathfrak{h}$ be the direct sum vector space. Consider the Nijenhuis-Richardson bracket $[~,~]_\mathsf{NR}$ on the graded space $\oplus_{n \geq 1} \mathrm{Hom}(\wedge^n W, W)$ of skew-symmetric multilinear maps given by
\begin{align*}
[f, g ]_\mathsf{NR} = f \diamond g - (-1)^{(m-1)(n-1)}~ g \diamond f, ~~~ \text{ where }
\end{align*}
\begin{align*}
(f \diamond g ) (w_1, \ldots, w_{m+n-1}) =  \sum_{\sigma \in \mathbb{S}_{(n, m-1)}} (-1)^\sigma ~ f \big(  g (w_{\sigma (1)}, \ldots, w_{\sigma (n)}), w_{\sigma (n+1)}, \ldots, w_{\sigma (m+n-1)}   \big).
\end{align*}
Then similar to the associative case, the graded space $\oplus_{n \geq 0} \mathrm{Hom}(\wedge^n \mathfrak{h}, \mathfrak{g})$ carries a graded Lie bracket
\begin{align}\label{derived-lie}
\{ \! [ P, Q ] \! \} := (-1)^m ~[[\pi_\mathfrak{g} + \rho, P]_\mathsf{NR}, Q ]_\mathsf{NR}, ~ \text{ for } P \in \mathrm{Hom}(\wedge^m \mathfrak{h}, \mathfrak{g}),~ Q \in \mathrm{Hom}(\wedge^n \mathfrak{h}, \mathfrak{g}).
\end{align}
The explicit formula of the bracket (\ref{derived-lie}) can be found in \cite[Equation (5)]{tang}. On the other hand, the differential $\delta_{ - \lambda \pi_\mathfrak{h}} := - [ \lambda \pi_\mathfrak{h}, - ]_\mathsf{NR}$ restricts to a differential $\delta$
on the graded space $\oplus_{n \geq 0} \mathrm{Hom}(\wedge^n \mathfrak{h}, \mathfrak{g})$. Moreover, we get that $\big( \oplus_{n \geq 0} \mathrm{Hom}(\wedge^n \mathfrak{h}, \mathfrak{g}) , \{ \! [ ~, ~ ] \! \}, \delta \big)$ is a differential graded Lie algebra.

\begin{thm}
Let $\mathfrak{g}$ and $\mathfrak{h}$ be two Lie algebras and $\mathfrak{g}$ acts on $\mathfrak{h}$ by a Lie algebra homomorphism $\rho : \mathfrak{g} \rightarrow \mathrm{Der}(\mathfrak{h})$. A linear map $T : \mathfrak{h} \rightarrow \mathfrak{g}$ is a $\lambda$-weighted relative Rota-Baxter operator if and only if $T$ is a Maurer-Cartan element in the differential graded Lie algebra $\big( \oplus_{n \geq 0} \mathrm{Hom}(\wedge^n \mathfrak{h}, \mathfrak{g}) , \{ \! [ ~, ~ ] \! \}, \delta \big)$.
\end{thm}

\begin{proof}
From the explicit formula \cite[Equation (5)]{tang}  of the bracket $\{ \! [ ~, ~ ] \! \}$, we get that
\begin{align*}
\{ \! [ T, T ] \! \} (u, v) = 2 \big( T (\rho (Tu) v) - T (\rho (Tv)u) - [Tu, Tv]_\mathfrak{g}  \big).
\end{align*}
On the other hand, We have $(\delta T)(u, v) = T (\lambda ~ [u, v]_\mathfrak{h}),$ for $u, v \in \mathfrak{h}$. Hence $T$ satisfies $\delta T + \frac{1}{2} \{ \! [ T, T ] \! \} = 0 $ if and only if $T$ is a  $\lambda$-weighted relative Rota-Baxter operator.
\end{proof}

It follows from the above theorem that a $\lambda$-weighted relative Rota-Baxter operator $T$ induces a differential $\delta_T := \delta + \{ \! [ T, ~ ] \! \}$ on the graded space  $\oplus_{n \geq 0} \mathrm{Hom}(\wedge^n \mathfrak{h}, \mathfrak{g})$. The corresponding cohomology groups are called the cohomology of $T$, denoted by $H^\ast_T (\mathfrak{h}, \mathfrak{g}).$

\medskip

The following result is similar to the associative case.

\begin{thm}
Let $T : \mathfrak{h} \rightarrow \mathfrak{g}$ be a $\lambda$-weighted relative Rota-Baxter operator. Then 
\begin{itemize}
\item[(i)] the triple $\big(  \oplus_{n \geq 0} \mathrm{Hom}(\wedge^n \mathfrak{h}, \mathfrak{g}) , \{ \! [ ~, ~ ] \! \}, \delta \big)$ is a differential graded Lie algebra.
\item[(ii)] For any linear map $T' : \mathfrak{h} \rightarrow \mathfrak{g}$, the sum $T + T'$ is a $\lambda$-weighted relative Rota-Baxter operator if and only if $T'$ is a Maurer-Cartan element in the differential graded Lie algebra $\big(  \oplus_{n \geq 0} \mathrm{Hom}(\wedge^n \mathfrak{h}, \mathfrak{g}) , \{ \! [ ~, ~ ] \! \}, \delta_T \big)$.
\end{itemize}
\end{thm}

\medskip

\medskip

In the following, we will show that the cohomology of $T$ can be described in terms of Chevalley-Eilenberg cohomology. We define a map $\rho^T : \mathfrak{h} \rightarrow \mathrm{End}(\mathfrak{g})$ by
\begin{align*}
\rho^T (u) (x) := T (\rho (x) u) + [Tu, x]_\mathfrak{g}, ~ \text{ for } u \in \mathfrak{h}, x \in \mathfrak{g}. 
\end{align*}

\begin{prop}
The map $\rho^T :  \mathfrak{h} \rightarrow \mathrm{End}(\mathfrak{g})$  defines a representation of the Lie algebra $(\mathfrak{h}, [~,~]_T)$ on the vector space $\mathfrak{g}$.
\end{prop}

\begin{proof}
For any $u, v \in \mathfrak{h}$ and $x \in \mathfrak{g}$, we have
\begin{align*}
&[\rho^T (u), \rho^T (v) ]  x \\
&= \rho^T (u) \rho^T (v) x  - \rho^T (v) \rho^T (u) x \\
&= \rho^T (u) \big( T (\rho (x) v) + [ Tv, x]_\mathfrak{g}  \big) -\rho^T (v) \big( T (\rho (x) u) + [ Tu, x]_\mathfrak{g}  \big) \\
&= T \big(   (\rho T (\rho (x) v)) u  )  \big) + [Tu, T (\rho (x) v) ]_\mathfrak{g} + T \big( \rho ([Tv, x]_\mathfrak{g}) u  \big) + [Tu, [Tv, x]_\mathfrak{g} ]_\mathfrak{g} \\
& \qquad \qquad - T \big(   (\rho T (\rho (x) u)) v  )  \big) - [Tv, T (\rho (x) u) ]_\mathfrak{g} - T \big( \rho ([Tu, x]_\mathfrak{g}) v  \big) - [Tv, [Tu, x]_\mathfrak{g} ]_\mathfrak{g} \\
&=  T \big(   (\rho T (\rho (x) v)) u  )  \big) +
 T \big(  \rho (Tu) \rho(x) v - \rho (T   (\rho(x) v) )u + \lambda~[u, \rho(x) v]_\mathfrak{h}  \big)
  + T \big( \rho ([Tv, x]_\mathfrak{g}) u  \big) + [Tu, [Tv, x]_\mathfrak{g} ]_\mathfrak{g} \\
&-  T \big(   (\rho T (\rho (x) u)) v  )  \big) -
 T \big(  \rho (Tv) \rho(x) u - \rho (T   (\rho(x) u) )v  + \lambda~[v, \rho(x) u]_\mathfrak{h}  \big)
  - T \big( \rho ([Tu, x]_\mathfrak{g}) v  \big) - [Tv, [Tu, x]_\mathfrak{g} ]_\mathfrak{g}  \\
 &= T \big(   \rho (Tu) \rho (x) v \big) + T \big(  \rho ([Tv, x]_\mathfrak{g}) u \big) - T \big(  \rho (Tv) \rho (x) u \big) - T \big(  \rho ([Tu, x]_\mathfrak{g}) v \big) \\
& \qquad \qquad + \lambda T ([u, \rho (x) v]_\mathfrak{h}) - \lambda T ([v, \rho(x) u]_\mathfrak{h}) + [[Tu, Tv]_\mathfrak{g}, x ]_\mathfrak{g} \\
&= - T \big( \rho (x) \rho (Tv) u \big) + T \big( \rho (x) \rho (Tu) v  \big) + \lambda T (\rho (x) [u, v]_\mathfrak{h}) + [T [u, v]_T, x]_\mathfrak{g} \\
&= T \big(  \rho(x) [u, v]_T \big) + [T [u, v]_T, x]_\mathfrak{g}  \\
&= \rho^T ([u, v]_T)x.
\end{align*}
This shows that $\rho^T$ is a representation. Hence the proof.
\end{proof}

Let $T : \mathfrak{h} \rightarrow \mathfrak{g}$ be a $\lambda$-weighted relative Rota-Baxter operator. Then it follows from the above proposition that one may consider the Chevalley-Eilenberg cohomology of the Lie algebra $(\mathfrak{h}, [~,~]_T)$ with coefficients in the representation $(\mathfrak{g}, \rho^T)$. More precisely, the $n$-th cochain group is $C^n_\mathsf{CE} (\mathfrak{h}, \mathfrak{g}) = \mathrm{Hom}(\wedge^n \mathfrak{h}, \mathfrak{g})$, for $n \geq 0$, and the coboundary map $\delta_\mathsf{CE} : C^n_\mathsf{CE} (\mathfrak{h}, \mathfrak{g}) \rightarrow C^{n+1}_\mathsf{CE} (\mathfrak{h}, \mathfrak{g})$ given by
\begin{align*}
&(\delta_\mathsf{CE} f)(u_1, \ldots, u_{n+1}) \\
&= \sum_{i=1}^{n+1} (-1)^{i+1} ~   T \big( \rho (f(u_1, \ldots, \widehat{u_i}, \ldots, u_{n+1})) u_i \big) + \sum_{i=1}^{n+1} (-1)^{i+1} ~ [ Tu_i, f(u_1, \ldots, \widehat{u_i}, \ldots, u_{n+1})]_\mathfrak{g}  \\
&+ \sum_{i < i} (-1)^{i+j} ~ f \big( \rho (Tu_i) u_j - \rho (Tu_j) u_i + \lambda [u_i, u_j]_\mathfrak{h}, u_1, \ldots, \widehat{u_i}, \ldots, \widehat{u_j}, \ldots, u_{n+1}  \big).
\end{align*}
The corresponding cohomology groups are denoted by $H^\ast_\mathsf{CE}(\mathfrak{h}, \mathfrak{g})$. Moreover, a direct computation says that
\begin{align*}
\delta_T f = (-1)^n ~\delta_\mathsf{CE} f, ~ \text{ for } f \in \mathrm{Hom} (\wedge^n \mathfrak{h}, \mathfrak{g}).
\end{align*}
Hence the cohomology $H^\ast_T (\mathfrak{h}, \mathfrak{g})$ is isomorphic to the Chevalley-Eilenberg cohomology $H^\ast_\mathsf{CE}(\mathfrak{h}, \mathfrak{g})$.

\begin{remark}
Similar to Section \ref{sec-4}, one may study deformations of $\lambda$-weighted relative Rota-Baxter operators on Lie algebras. The main results are similar to those given in Section \ref{sec-4}. 
\end{remark}

\subsection{Relation with associative case}

Let $A$ be an associative algebra. Then there is a Lie algebra structure on $A$ given by the commutator bracket $[a, b]_c = a \cdot_A b - b \cdot_A a$, for $a, b \in A$. We denote this Lie algebra by $A_c$. If $A$ and $B$ are two associative algebras and $B$ is an associative $A$-bimodule, then it is easy to see that the Lie algebra $A_c$ acts on the Lie algebra $B_c$ by a Lie algebra homomorphism
\begin{align*}
\rho : A_c \rightarrow \mathrm{Der}(B_c), ~ \rho (a)(u) = a \cdot u - u \cdot a, ~ \text{ for } a \in A_c, u \in B_c.
\end{align*}

With the above notations, we have the following.

\begin{prop}\label{skew-prop}
\begin{itemize}
\item[(i)] Let $T \in \mathsf{rRB}_\lambda (B, A)$. Then $T \in \mathsf{rRB}_\lambda (B_c, A_c)$.
\item[(ii)] Let $T, T' \in \mathsf{rRB}_\lambda (B, A)$. If $(\phi, \psi)$ is a morphism from $T$ to $T'$ in the category $\mathsf{rRB}_\lambda (B, A)$, then $(\phi, \psi)$ is also a morphism from $T$ to $T'$ in the category  $\mathsf{rRB}_\lambda (B_c, A_c)$.
\end{itemize}
\end{prop}

\begin{proof}
(i) For any $u, v \in B_c$, we have
\begin{align*}
[T(u), T(v)]_{A_c} =~& T(u) \cdot_A T(v) - T(v) \cdot_A T(u) \\
=~& T (T(u) \cdot v + u \cdot T(v) + \lambda ~ u \cdot_B v) - T (T(v) \cdot u + v \cdot T(u) + \lambda ~ v \cdot_B u) \\
=~& T ( \rho (Tu) v - \rho (Tv) u + \lambda ~ [u, v]_{B_c} ).
\end{align*}
This proves that $T \in \mathsf{rRB}_\lambda (B_c, A_c).$

(ii) Since $\phi : A \rightarrow A$ (resp. $\psi : B \rightarrow B$) is an algebra morphism, it follows that $\phi : A_c \rightarrow A_c$ (resp. $\psi: B_c \rightarrow B_c$) is a Lie algebra morphism. Moreover, we have $\phi \circ T = T' \circ \psi$. Finally,
\begin{align*}
\psi (\rho (a) u) = \psi (a \cdot u - u \cdot a) = \phi (a) \cdot \psi (u) - \psi (u) \cdot \phi (a) = \rho (\phi (a)) \psi (u).
\end{align*}
This shows that $(\phi, \psi)$ is a morphism from $T$ to $T'$ in the category $\mathsf{rRB}_\lambda (B_c, A_c)$.
\end{proof}

Let $T \in \mathsf{rRB}_\lambda (B, A)$.
In the following, we find the relation between the cohomologies $H^\ast_T (B,A)$ and $H^\ast_T (B_c, A_c)$. We start by recalling the following standard result. Let $A$ be an associative algebra and $M$ be an $A$-bimodule. Then $M$ can be considered as a representation of the Lie algebra $A_c$ with the action $\rho : A_c \rightarrow \mathrm{End}(M)$ given by $\rho (a)(m) = a \cdot m - m \cdot a$, for $a \in A_c$ and $m \in M$. We denote this representation by $M_c$.

\begin{prop}\label{sn-map}
The collection $\{ S_n \}_{n \geq 0}$ of maps $S_n : \mathrm{Hom}(A^{\otimes n}, M) \rightarrow \mathrm{Hom}(\wedge^n A_c, M_c)$ defined by
\begin{align*}
S_n (f) (a_1, \ldots, a_n ) = \sum_{\sigma \in \mathbb{S}_n} (-1)^\sigma ~ f (a_{\sigma(1)}, \ldots, a_{\sigma (n)})
\end{align*}
is a morphism from the Hochschild cochain complex of $A$ with coefficients in the $A$-bimodule $M$ to the Chevalley-Eilenberg cochain complex of $A_c$ with coefficients in the representation $M_c$.
\end{prop}

\medskip

Let $A$ and $B$ be two associative algebras and $B$ be an associative $A$-bimodule. Let $T \in \mathsf{rRB}_\lambda (B, A)$ be a $\lambda$-weighted relative Rota-Baxter operator. Then we have seen in Subsection \ref{subsec-hoch} that the cohomology of $T$ is isomorphic to the Hochschild cohomology of $(B, \cdot_T)$ with coefficients in the bimodule $(A, l^T, r^T)$.

On the other hand, we have from Proposition \ref{skew-prop} that $T \in \mathsf{rRB}_\lambda (B_c, A_c)$ . The cohomology of $T$ (as a $\lambda$-weighted relative Rota-Baxter operator on Lie algebra) is isomorphic to the Chevalley-Eilenberg cohomology of $(B, [~,~]_T)$ with coefficients in the representation $(A, \rho^T)$, where
\begin{align*}
[u, v]_T :=  u \cdot_T v - v \cdot_T u ~~~~ \text{ and } ~~~~
\rho^T (u) (a) := l^T_u (a) - r^T_u (a).
\end{align*}

Thus, as a consequence of Proposition \ref{sn-map}, we get the following.

\begin{thm}
Let $A$ and $B$ be two associative algebras and $B$ be an associative $A$-bimodule. Let $T \in \mathsf{rRB}_\lambda (B, A)$ be a $\lambda$-weighted relative Rota-Baxter operator. Then the collection $\{ S_n \}_{n \geq 0}$ of maps 
\begin{align*}
S_n : \mathrm{ Hom}(B^{\otimes n}, A) \rightarrow \mathrm{Hom}(\wedge^n B, A), ~~~ S_n (f) (u_1, \ldots, u_n) = \sum_{\sigma \in \mathbb{S_n}} (-1)^\sigma ~ f (u_{\sigma(1)}, \ldots, f_{\sigma (n)})
\end{align*} 
induces a morphism of cohomologies from $H^\ast_T (B, A)$ to $H^\ast_T (B_c, A_c)$.
\end{thm}

\medskip

\noindent {\bf Acknowledgements.}  The research is supported by the fellowship of Indian Institute of Technology (IIT) Kanpur. The author thanks the Institute for support.

\medskip

\medskip


\end{document}